\documentclass[12pt]{article}

\usepackage{graphicx}
\usepackage[all]{xypic}
\usepackage{amsmath}
\usepackage{amsfonts}
\usepackage{amssymb}
\usepackage{amsthm}

\usepackage[top=2cm, bottom=2cm, left=2cm, right=2cm]{geometry}

\newtheorem{thm}{Theorem}[section]
\newtheorem{mthm}{Theorem}

\newtheorem{crl}[thm]{Corollary}

\newtheorem{prp}[thm]{Proposition}

\newtheorem{lmm}[thm]{Lemma}

\newtheorem{mconj}[mthm]{Conjecture}

\newcommand {\tb}{\textbf}
\newcommand {\mb}{\mathbb}
\newcommand {\Z}{\mb Z}
\newcommand {\R}{\mb R}
\newcommand {\C}{\mb C}

\newcommand {\colim}{\textrm{colim}\ }

\newcommand {\ext}{\mathrm{Ext}}
\newcommand {\lra}{\longrightarrow}
\newcommand {\la}{\langle}
\newcommand {\ra}{\rangle}

\begin{document}

\title{Generalised geometric weak conjecture on spherical classes and non-factorisation of Kervaire invariant one elements}

\author{Hadi Zare
\thanks{I am grateful for a partial support that I have received from the University of Tehran.}
\\
School of Mathematics, Statistics, and Computer Science,\\
College of Science, University of Tehran, Tehran, Iran \textup{14174}\\
\textit{hadi.zare@ut.ac.ir}
}

\maketitle

\date{}

\begin{abstract}
This paper is on the Curtis conjecture which states that the image of the Hurewicz homomorphism $h:{_2\pi_*}Q_0S^0\to H_*(Q_0S^0;\Z/2)$, in positive degrees, only consists of Hopf invariant one and Kervaire invariant one elements. We show that the image of the Hurewicz homomorhism
$h:\pi_*Q_0S^0\to H_*(Q_0S^0;\Z)$, when restricted to product of positive dimensional elements, is determined by $\Z\{h(\eta^2),h(\nu^2),h(\sigma^2)\}$. Localised at $p=2$, this proves a geometric version of a result of Hung and Peterson for the Lannes-Zarati homomorphism where they show that the Lannes-Zarati homomorphism $\varphi_k$ vanishes on decomposable classes for $k>2$; we argue that our result provides a stronger evidence for Curtis conjecture. We apply this to show that, for $p=2$ and $G=O(1)$ or any prime $p$ and $G$ any compact Lie group with Lie algebra $\mathfrak{g}$ so that $\dim\mathfrak{g}>0$, the composition
$${_p\pi_*}Q\Sigma^{n\dim\mathfrak{g}}BG_+^{\wedge n}\to {_p\pi_*}Q_0S^0\stackrel{h}{\to}H_*(Q_0S^0;\Z/p)$$
where $\Sigma^{n\dim\mathfrak{g}}BG_+^{\wedge n}\to S^0$ is the $n$-fold transfer, is trivial if $n>2$. Moreover, we show that for $n=2$, the image of the above composition vanishes on all elements of Adams filtration at least $1$, i.e. those elements of ${_2\pi_*^s}\Sigma^{n\dim\mathfrak{g}}BG_+^{\wedge n}$ represented by a permanent cycle $\ext_{A_p}^{s,t}(\widetilde{H}^*\Sigma^{n\dim\mathfrak{g}}BG_+^{\wedge n},\Z/p)$ with $s>0$, map trivially under the above composition. The case of $n>2$ of the above observation proves and generalises a geometric variant of the weak conjecture on spherical classes due to Hung, later on verified by Hung and Nam. We argue that, for the purpose of verifying Curtis conjecture, our result provides a stronger evidence, and proves more than their result. We also show that, for a compact Lie group $G$, Curtis conjecture holds if we restrict to the image of the $n$-fold transfer $\Sigma^{n\dim\mathfrak{g}}BG_+^{\wedge n}\to S^0$ with $n>1$. Finally, we show that the Kervaire invariant one elements $\theta_j\in{_2\pi_{2^{j+1}-2}^s}$ with $j>3$ do not factorise through the $n$-fold transfer $\Sigma^{n\dim\mathfrak{g}}BG_+^{\wedge n}\to S^0$ with $n>1$ for $G=O(1)$ or any compact Lie group as above, hence reproving and generalising a result of Minami.
\end{abstract}


\section{Introduction and statement of results}
This note is circulated around the Curtis conjecture; some of the observations here might be well known, but we don't know of any published account. For a space $X$, we write $QX=\colim\Omega^i\Sigma^iX$ and $Q_0X$ for its base point component; $Q_0S^0$ is the base point component of $QS^0=\colim \Omega^iS^i$ corresponding to stable maps $S^0\to S^0$ of degree $0$. By Theorem \ref{main0} below, the Hurewicz homomorphism $h:\pi_*QS^0\to H_*(QS^0;\Z)$, as well as its $2$-local version ${_2\pi_*}Q_0S^0\to H_*(Q_0S^0;\Z)$, is not a ring homomorphism. The conjecture then reads as follows (see \cite[Theorem 7.1]{Curtis} and \cite{Wellington} for more discussions).
\begin{mconj}[The Curtis Conjecture]\label{conjecture}
The image of $h:{_2\pi_*}Q_0S^0\lra H_*(Q_0S^0;\Z/2)$, in positive degrees, is the $\Z/2\{h(\eta),h(\nu),h(\sigma),h(\theta_j)\}$, the $\Z/2$-vector space generated by $h(\eta)$, $h(\nu)$, $h(\sigma)$ and $h(\theta_j)$ where $\theta_j\in{_2\pi_{2^{j+1}-2}^s}$ denotes a Kerviare invariant one element.
\end{mconj}

Of course, after Hill-Hopkins-Ravenel \cite{HillHopkinsRavenel}, the conjecture implies that the image of $h$ is finite. Note that given $f,g\in{_2\pi_*}Q_0S^0$ with $h(f)\neq 0$ and $h(g)=0$ then $h(f+g)=h(f)\neq 0$. Also, note that if $\alpha:S^0\not\lra S^0$ is any map of odd degree, then $h(\alpha f)=h(f)$. These hopefully will justify the way that we have stated our results. Here and throughout the paper, we write $\pi_*$ and $\pi^s_*$ for homotopy and stable homotopy respectively, and ${_p\pi_*},{_p\pi_*^s}$ for their $p$-primary components, respectively. We shall use $f_*$ for the mapping induced in homology, where $f$ is a mapping of spaces or stable complexes. We write $\eta,\nu,\sigma$ for the well known Hopf invariant one elements.  For a (virtual) vector bundle $\xi\to X$, we shall write $X^\xi$ for its Thom (specturm) space. We write $\mathfrak{g}$ for the Lie algebra of a Lie group $G$ and $A_p$ for the mod $p$ Steenrod algebra. The notations $X_+$ refers to a space $X$ with an added base point. We use $X^{\times n}$ and $X^{\wedge n}$ to denote $n$-fold Cartesian and smash products of $X$ with itself, respectively.\\

The source of the Hurwicz homomorphism is computed by the Adams spectral sequence (ASS) $\ext_A(\Z/2,\Z/2)\Rightarrow{_2\pi_*^s}\simeq{_2\pi_*}Q_0S^0$ whereas its target sits inside $H_*(QS^0;\Z/2)\simeq\Z/2[Q^I[1];I\textrm{ admissible}]$ and is related to the Dyer-Lashof algebra. By results of Madsen
\cite[Corollary 3.3]{Madsen} the dual of Dyer-Lashof algebra and Dickson algebra are related. Moreover, the ASS and Dickson algebra are related through the Lannes-Zarati homomorphism $\varphi_k:\ext_A^{k,k+i}(\Z/2,\Z/2)\rightarrow(\Z/2\otimes_A D_k)^*_i$. Since, the Hopf invariant one and Kervaire invariant one elements are detected in the $1$-line and $2$-line of the ASS, respectively, Hu'ng and Peterson \cite[Conjecture 1.2]{HungPeterson} then were led to the following conjecture.
\begin{mconj}[Algebraic conjecture on spherical classes]
$\varphi_k=0$ for all $k>2$.
\end{mconj}
Initially, they believed that this implies Conjecture \ref{conjecture}. However, as pointed out to the author by Nick Kuhn, theoretically, it is possible to have the vanishing of the Lannes-Zarati homomorphism on a permanent cycle $c$ converging to an element $f$ with $h(f)\neq 0$ (see also \cite{Hung-erratum}). Hence, the above conjecture does not necessarily imply \ref{conjecture}, although it provides a good evidence. As far as I know, both conjectures remain open to this day. Later on Hu'ng introduced a weak version of the above conjecture \cite{Hung-weakconjecture}. Let $\varphi_k:\mathrm{Ext}^{k,k+i}_A(\Z/2,\Z/2)\to (\Z/2\otimes D_k)_i^*$ be the Lannes-Zarati homomorphism, and $Tr_k:\Z/2\otimes_{GL_k}PH_i(BV_k)\to \mathrm{Ext}^{k,k+i}_A(\Z/2,\Z/2)$ be Singer's algebraic transfer where $V_k=(\Z/2)^{\times k}$ and $P(-)$ is the primitive submodule functor. The conjecture then reads as follows
\cite[Conjecture 1.3]{Hung-weakconjecture}(see also
\cite[Conjecture 1.3]{Hung-algebraictransfer}).
\begin{mconj}[Weak algebraic conjecture on spherical classes]
The composition
$$\varphi_k\circ Tr_k:\Z/2\otimes_{GL_k}PH_i(BV_k)\to (\Z/2\otimes D_k)_i^*$$
is trivial for $k>2$ \cite[Conjecture 1.3]{Hung-algebraictransfer}.
\end{mconj}
Noting that $BV_k$ is just $\R P^{\times k}$ and that $Tr_k$ is meant to be an algebraic version of the $n$-fold transfer $\R P^{\wedge n}_+\to S^0$, as noted by Hung in \cite[Conjecture 1.3]{Hung-weakconjecture}, the above conjecture aims to imply that the Hurewicz homomorphism $h:{_2\pi_*}Q_0S^0\to H_*(Q_0S^0;\Z/2)$ vanishes on the image of the $n$-fold transfer. Again, as noted above, the above considerations on the Lannes-Zarati homomorphism, does not allow us to make such a deduction. That is, although the Weak algebraic conjecture was later established by Hu'ng and Nam \cite[Main Theorem]{HungNam}, we cannot use it to deduce the above claim on the image of the Hurewicz homomorphism when restricted to the image of the $n$-fold transfer. Let's conclude that most of the work cited above, take place in the realm of algebra, and are heavily related to the invariant theory of the $GL_k$, Dickson algebra, as well the famous hit problem for $BV_k$.\\

We employ geometric methods to prove some facts on the Hurewicz homomorphism, and try to work with the actual homomorphism rather than any algebraic approximation to it. We first consider the Hurewicz homomorphism on decomposable elements in ${\pi_*^s}$. We prove integral results so that $p$-local results are derived in a natural way.

\begin{mthm}\label{main0}
(i) For $i,j>0$, consider the composition
$${\pi_i}Q_0S^0\otimes {\pi_j}Q_0S^0\lra {\pi_{i+j}}Q_0S^0\stackrel{h}{\lra} H_{i+j}(Q_0S^0;\Z)$$
where the first arrow is the product in $\pi_*^s$. Then $h(fg)\neq 0$ only if $f$ and $g$ live in the same grading, and both are detected by the unstable Hopf invariant.\\
(ii) For $i,j>0$, consider the composition
$${_2\pi_i}Q_0S^0\otimes {_2\pi_j}Q_0S^0\lra {_2\pi_{i+j}}Q_0S^0\stackrel{h}{\lra} H_{i+j}(Q_0S^0;\Z/2).$$
Then $h(fg)\neq 0$ only if $f=g$ with $f=\eta,\nu,\sigma$ or odd multiples of these elements, i.e. the image of this composition only consists of Kerviare invariant one elements $h(\eta^2),h(\nu^2),h(\sigma^2)$. Here, the first arrow on the left is the multiplication on the stable homotopy ring. Also, the image of the composite
$$\la f:f=\eta,\nu,\sigma\ra\rightarrowtail{_2\pi_*}Q_0S^0\stackrel{h}{\lra}H_*(Q_0S^0;\Z/2)$$
only consists of the Hurewicz image of the Hopf invariant one elements $\eta,\nu,\sigma$, and the Kervaire invariant one elements $\eta^2,\nu^2,\sigma^2$.\\
(iii) Suppose $E$ is a $CW$-spectrum, and let $\Omega^\infty E=\colim\Omega^nE_n$. Then for the composition
$$h^\circ:\pi_iQ_0S^0\otimes\pi_j\Omega^\infty E\to\pi_{i+j}\Omega^\infty E\lra H_*(\Omega^\infty E;\Z)$$
with $i>0$, we have
$$\begin{array}{ccl}
i<j &\Rightarrow& h^\circ(\alpha\circ f)=0,\\
i>j &\Rightarrow& e_*^{i-j+1}h^\circ(\alpha\circ f)=0.
\end{array}$$
Moreover, if $i=j$ and $h^\circ(\alpha\circ f)\neq 0$ then $\alpha$ is detected by the unstable Hopf invariant; in particular at the prime $2$, $\alpha\in\{\eta,\nu,\sigma\}$.\\
Here, $\circ$ is the product coming from the composition of stable maps $S^{i+j}\to S^j$ and $S^j\to X$ turning $\pi_*\Omega^\infty E$ into a left $\pi_*^s$-module.
\end{mthm}

Let's note part (ii) of the above theorem, might seem as an analogue of \cite[Proposition 5.4]{HungPeterson} for the Lannes-Zarati homomorphism which shows that $\varphi_k$ vanishes on decomposable classes for $k>2$. However, again, due to the above considerations on $\varphi_k$, for the purpose of studying the Hurewicz homomorphism, our result provides a stronger evidence for Conjecture \ref{conjecture}. Our first application of this theorem is to provide some necessary conditions for an element $f\in{_2\pi_*^s}$ to map nontrivially under $h$ which we state as follows.

\begin{mthm}\label{necessary}
Suppose $f\in{_p\pi_*^s}$ is of Adams filtration $>2$ which maps nontrivially under $h:{_p\pi_*}Q_0S^0\to H_*(Q_0S^0;\Z/p)$. Then, $f$ is not a decomposable element in ${_2\pi_*^s}$.
\end{mthm}

The proof of the above theorem is very short, so we include it here.

\begin{proof}
First, note that the elements of Adams filtration $1$ or $2$ in ${_p\pi_*^s}$ are known (see for example the tables of \cite{Toda} and \cite{Ravenel-Greenbook}). 
Suppose $f$ is a sum of decomposable terms. Since $f$ is Adams filtration at least $3$, then if written as a sum of decomposable terms in ${_p\pi_*^s}$, it cannot involve terms such as $\eta^2,\nu^2$ or $\sigma^2$. It is known that the only element, in one of the dimensions $2$, $6$, or $14$ is one of higher filtration, namely $\kappa\in{_2\pi_{14}^s}$, which is known to map trivially under $h$ (see also \cite[Lemma 3.7]{Z-extensions} for a proof). Since $h$ is a homomorphism of graded modules, hence, by Theorem \ref{main0}, $h(f)=0$.
\end{proof}

We are now in a position to state the first application of the above analysis which is a geometric, stronger, and generalised version of the weak algebraic conjecture on spherical classes. We have the following.

\begin{mthm}\label{n-foldtransfer}
Let $p$ be any prime. For $G=O(1)$ and $p=2$ or $G$ any compact Lie group with $\dim\mathfrak{g}>0$ and $p$ any prime, consider the the $n$-fold transfer $BG_+^{\wedge n}\to S^0$ associated to the embedding $1<G^{\times n}$
$${_p\pi_*^s}\Sigma^{n\dim\mathfrak{g}}BG^{\wedge n}_+\lra {_p\pi_*}Q_0S^0\lra H_*(Q_0S^0;\Z/p).$$
The following statements then hold.\\
(i) For $n>2$, the above composition is trivial.\\
(ii) For $n=2$, the above composition vanishes on all elements of Adams filtration at least $1$, i.e. those elements of ${_2\pi_*^s}\Sigma^{n\dim\mathfrak{g}}BG_+^{\wedge n}$ represented by a permanent cycle $\ext_{A_p}^{s,t}(\widetilde{H}^*\Sigma^{n\dim\mathfrak{g}}BG_+^{\wedge n},\Z/p)$ with $s>0$, map trivially under the above composition. Moreover, the image of this composition on the elements of Adams filtration $0$ is contained within $\Z/p\{h(\eta^2),h(\nu^2),h(\sigma^2)\}$. Furthermore, in the case of $G=O(1)$ and $p=2$ the image of the above composition is precisely given by $\Z/2\{h(\eta^2),h(\nu^2),h(\sigma^2)\}$
\end{mthm}

The above theorem proves a geometric version of the weak conjecture on spherical classes. Due to the above consideration on the relation between vanishing of Lannes-Zarati homomorphism and Hurewicz homomorphism, our theorem provides a stronger result for Curtis conjecture \ref{conjecture}. Also noted that we have stated our theorem for any compact Lie group, hence a generalised form of the conjecture of Hu'ng. We have to add that this is still far from a complete proof Conjecture \ref{conjecture}. By the generalised Kahn-Priddy Theorem (see\cite{Finkelstein}, \cite{Kuhn-extended}) there is a map $\R P\times QD_2\R P\to Q\R P$ which induces an epimorphism in ${_2\pi_*}$. Combined with Kahn-Priddy's theorem \cite[Theorem 3.1]{Kahn-Priddy}, at the prime $2$, we obtain a sequence of epimorphisms
$${_2\pi_n}QD_2\R P\to {_2\pi_n}Q\R P\to {_2\pi_n}Q_0S^0$$
for $n>1$. Now, the space $\R P\wedge\R P$ is just the first filtration of the space $D_2\R P=S^{\infty}\ltimes_{\Sigma_2}\R P^{\wedge 2}$ which is filtered by $D_2^i\R P=S^i\ltimes_{\Sigma_2}\R P^{\wedge 2}$ with $D_2^0\R P=\R P^{\wedge 2}$ \cite{Eccles-browder}. The conjecture can be resolved if one can say something about the homotopy of $QD_2\R P\to Q\R P\to Q_0S^0$.\\

As an application, we provide a generalisation of Minami's results on the factorisation of Kervaire invariant one elements.
\begin{mthm}\label{gen.minami}
For $j>3$ and $n>1$, the Kervaire invariant one elements $\theta_j\in{_2\pi_{2^{j+1}-2}^s}$ do not factor through the $n$-fold transfer $\Sigma^{n\dim\mathfrak{g}}BG^{\wedge n}_+\to S^0$ where $G=O(1)$ or $G$ any compact Lie group with $\dim\mathfrak{g}>0$.
\end{mthm}

We invite the reader to compare our method of proof to the original method of Minami using $BP$-based ASS arguments.\\

\tb{Acknowledgements.} I am grateful to Nick Kuhn for helpful discussions on the conjecture and in particular on the Lannes-Zarati homomorphism, to Nguy\^en H.V. {Hu'ng} for drawing my attention to their joint work with F. Peterson and in particular to \cite[Proposition 5.4]{HungPeterson}, to Geoffrey Powell and Gerald Gaudens for an invitation to LAERMA during October 2014 as well as to the University of Tehran, IPM and LAERMA whom their partial support made this visit possible.

\section{Preliminaries: The product in $\pi_*^s$}
The material here are classic for which our main references are \cite{BarrattHilton} and \cite[Chapter 8]{May-G}. Given stable maps $f:S^n\to S^0$ and $g:S^m\to S^0$ we may compose them in two way. First, we have $f\wedge g:S^{n+m}\to S^0$. Second, we may consider their composition product $S^{n+m}\stackrel{\Sigma^m f}{\to}S^m\stackrel{g}{\to} S^0$. It is known that the two products agree in the stable homotopy ring \cite[Theorem 3.2]{BarrattHilton}, \cite[Lemma 8.7]{May-G}. The following then is obvious -
\begin{prp}
An element in ${_2\pi_*^s}$ is $\wedge$-decomposable if and only if it is $\circ$-decomposable.
\end{prp}
For this reason, we use the two products interchangeably; for $f\in\pi_i^s$ and $g\in\pi_j^s$ we allow ourselves to think of $gf$ as $S^{i+j}\to S^j\to S^0$. Moreover, since we choose to work with $\pi_*QS^0\simeq\pi_*^s$, we will consider stale adjoints of above stable compositions, such as $S^{i+j}\to QS^j\to QS^0$ in order to realise $gf$ as an element in $\pi_*QS^0$. Let's conclude by recalling that $\pi_*^s$ is a commutative ring \cite{Barratt-Toda}.

\section{Hurewicz homomorphism and products}
The aim of this section is to provide a simple proof for Theorem \ref{main0}. We begin with the following.

\begin{lmm}\label{ij}
Let $f\in{\pi_i}Q_0S^0,g\in{\pi_j}Q_0S^0$ with $i,j>0$ and $i\neq j$. Then $h(fg)=0$.
\end{lmm}

\begin{proof}
The product $gf$ in ${\pi_*^s}$ is determined by the composition of stable maps $S^{i+j}\stackrel{f}{\to}S^j\stackrel{g}{\to}S^0$ which we wish to compute homology of its stable adjoint $S^{i+j}\to QS^j\to QS^0$. For $i<j$ the map $S^{i+j}\lra S^j$ is in the stable range, so it can be taken as a genuine map. More precisely, by Freudenthal suspension theorem, $S^{i+j}\to QS^j$ does factorise as $S^{i+j}\to S^j\to QS^j$ where the part $S^{i+j}\to S^j$ is not necessarily unique. The stable adjoint of $gf$ then maybe viewed as $S^{i+j}\stackrel{f}{\lra}S^j\stackrel{g}{\lra}Q_0S^0$
which is trivial in homology for dimensional reasons. The indeterminacy in choosing the pull back $S^{i+j}\to S^j$ is irrelevant here and will not effect the dimensional reason. Hence, $h(gf)=0$. The case $i>j$ is similar, noting that $\pi_*^s=\pi_*Q_0S^0$ is commutative.
\end{proof}

Note that the above result hold integrally, and consequently on ${_p\pi_*^s}$ for any prime $p$. According to the above lemma, $h(fg)\neq 0$ may occur if $f,g\in{\pi_n}Q_0S^0$. Note that for $f\in{\pi_n^s}$ there exists $\widetilde{f}\in\pi_{2n+1}S^{n+1}$, not necessarily unique, which maps to $f$ under the stabilisation $\pi_{2n+1}S^{n+1}\lra\pi_n^s$. We shall say $f$ is detected by the unstable Hopf invariant, if $\widetilde{f}$ is detected by cup-squaring operation in its mapping cone, i.e. $g_{n+1}^2\neq 0$ in $H^*(C_{\widetilde{f}};\Z)$.

\begin{thm}\label{nilpotence}
Suppose $f,g\in\pi_nQ_0S^0$  with $h(fg)\neq 0$ with $h:\pi_*Q_0S^0\lra H_*(Q_0S^0;\Z)$ being the Hurewicz homomorphism. Then both $f$ and $g$ are detected by the unstable Hopf invariant.
\end{thm}

\begin{proof}
We are interested in the stable adjoint of $S^{2n}\stackrel{f}{\lra}S^n\stackrel{g}{\lra}S^0$ given by
$$S^{2n}\stackrel{f_n}{\lra} QS^n\stackrel{g}{\lra} Q_0S^0.$$
As noted above, for dimensional reasons, the mapping $f_n$ factors as $S^{2n+1}\stackrel{\widetilde{f}_n}{\lra}\Omega S^{n+1}\lra QS^n$ where $\widetilde{f}_n$, which is not necessarily unique, is the adjoint for an appropriate $\widetilde{f}$; the map $i:\Omega S^{n+1}\lra QS^n$ is the stablisation map. Hence, the stable adjoint of $fg$ can be seen as a composite
$$S^{2n}\lra \Omega S^{n+1}\lra QS^n\lra Q_0S^0.$$
Now, $h(fg)\neq 0$ implies that $h(\widetilde{f}_n)\neq 0$. This shows that $h(\widetilde{f}_n)=\lambda g_n^2$ for some nonzero $\lambda\in\Z$ where $H_*(\Omega S^{n+1}\simeq T(g_n)$ is the tensor algebra over an $n$-dimensional generator $g_n\in H_n(\Omega S^{n+1};\Z)$ by James' results \cite{Whitehead}.
On the other hand, it is well know that $h(\widetilde{f}_n)=\lambda g_n^2$ if and only if $g_{n+1}^2=\pm\lambda g_{2n+2}$ in $H^*(C_{\widetilde{f}};\Z)$, i.e. $f$ is detected by the unstable Hopf invariant (see for example \cite[Proposition 6.1.5]{Harper}). Similarly, $g$ is also detected by the unstable Hopf invariant.
\end{proof}

The result in integral case, implies the $p$-primary case. In particular, we have the following.

\begin{crl}
Suppose $f,g\in{_2\pi_n}Q_0S^0$  with $h(fg)\neq 0$ then both $f$ and $g$ are Hopf invariant one elements, i.e. $f,g=\eta,\nu,\sigma$. Here,  $h:{_2\pi_*}Q_0S^0\lra H_*(Q_0S^0;\Z/2)$ is the mod $2$ Hurewicz homomorphism.
\end{crl}

In order to complete the proof of Theorem \ref{main0} note that by above observations the nontrivial image of the Hurewicz homomorphism on the the ideal $\la f:f=\eta,\nu,\sigma\ra$ can only arise from elements $f,f^2$. Moreover, choosing $g:S^0\lra S^0$ to be a stable map of odd degree, $f=\eta,\nu,\sigma$, then $fg$ is an odd multiple of a Hopf invariant one element which we know map nontrivially under $h$.\\
For the case of $h^\circ$ we note that given $\alpha\in\pi_iQ_0S^0$ and $f\in\pi_j\Omega^\infty E$ the Hurewicz image of $f\circ \alpha$ is determined by the homology of the composition
$$S^{i+j}\stackrel{\alpha}{\lra} QS^j\stackrel{f}{\lra} \Omega^\infty E.$$
By similar reasons, if $i<j$ then by Fruedenthal's theorem $\alpha$ doesfactorise through some map $S^{i+j}\to S^j$, hence $h^\circ(\alpha\circ f)=0$. If $i=j$ then $\alpha$ admits a factorisation as $S^{2i}\to\Omega\Sigma S^i\to QS^j\to \Omega^\infty E$ which if is nontrivial in homology which imply that $\alpha$ is detected by the unstable Hopf invariant, so $\alpha\in\{\eta,\nu,\sigma\}$. Finally, for the case $i>j$, note that $(i-j+1)$-th adjoint of $\alpha:S^{i+j}\to QS^j$ is a map $S^{2i+1}\to QS^{i+1}$ which does factorise throughout $S^{2i+1}\to S^{i+1}$ so $e_*^{i-j+1}h^\circ(\alpha\circ f)=0$.


\section{Hurewicz homomorphism and $n$-fold transfer maps}\label{Hurewicz-transfer}
This section is devoted to the proof of Theorem \ref{n-foldtransfer}. Suppose $H<G$ are compact Lie group and consider the fibration $G/H\to BH\stackrel{p}{\to} BG$. Let $E$ be a space with a free action of $G$ and let $\mu_G=E\times_G \mathfrak{g}\to E/G$, $G$ acting on $\mathfrak{g}$ through its adjoint representation, be the associated bundle whose fibre is the Lie algebra of $\mathfrak{g}$ of $G$. Suppose $\alpha\to E/H$ is a (virtual) bundle. The transfer map then is a stable map between Thom spectra \cite{Milequi}, \cite{BeckerSchultz1} and \cite{Morisugi}
$$t_H^G:(E/G)^{\mu_G+\alpha}\to(E/H)^{\mu_H+p^*\alpha}.$$
This construction enjoys natural properties to expect. For instance, the transfer associated to the identity $1:G\to G$ induces the identity between Thom spectra. It follows that \cite[Note 1.14]{Milequi} for $H_1<G_1$ and $H_2<G_2$, with twisting bundles $\alpha_1\to E_1/G_1$ and $\alpha_2\to E_2/G_2$, we have
$$t_{H_1\times H_2}^{G_1\times G_2}=t_{H_1}^{G_1}\wedge t_{H_2}^{G_2}.$$
If we choose $H=1$ to be the trivial group, and choose $E$ to be a contractible free $G$-space, filtering $E$ by compact manifolds if necessary as noted by \cite{Milequi}, then by choosing $\alpha=-\mu_G$ we have a transfer map
$$BG^0\lra B1^{-p^*\mu_G}=B1^{-\R^{\dim\mathfrak{g}}}=S^{-\dim\mathfrak{g}}$$
which after $\dim\mathfrak{g}$ times suspension yields
$$\Sigma^{\dim\mathfrak{g}}BG_+\to S^0.$$
Replacing $G$ with $G^{\times n}$ and $1$ with $1^{\times n}$ yields a transfer map $\Sigma^{n\dim\mathfrak{g}}BG_+^{\wedge n}\to (S^0)^{\wedge n}$ which after composition with the ring spectrum multiplication of $S^0$, $(S^0)^{\wedge n}\to S^0$ yields the $n$-fold transfer map
$$\Sigma^{n\dim\mathfrak{g}}BG_+^{\wedge n}\to (S^0)^{\wedge n}\to S^0$$
of the abstract. We are interested in the cases with $G=O(1)^{\times n}$ and $G=(S^1)^{\times n}$ where we have
$$\R P^{\wedge n}_+\to S^0,\ \ \Sigma^n\C P_+\to S^0$$
of which the case $G=(S^1)^{\times n}$ has been extensively studied (see for example \cite{Baker-doubletransfer}, \cite{Imaoka-factorisation}, \cite{BCGHRW}). In the case $n=1$, $G=O(1)$, after Kahn-Priddy \cite[Theorem 3.1]{Kahn-Priddy}, we know that the transfer $\lambda:\R P_+\to S^0$ induces an epimorphism in ${_2\pi_*^s}$. Before proceeding with the proof of Theorem \ref{n-foldtransfer}, we need to sort out a little technical issue.
\begin{lmm}\label{trans-decomposable}
Suppose $n>1$ and $p$ is a prime. Let $G=O(1)$ and $p=2$, or $G$ any compact Lie group with $\dim\mathfrak{g}>0$. Then, at the prime $p$, the image of the $n$-fold transfer map $\Sigma^{n\dim\mathfrak{g}}BG_+^{\wedge n}\to S^0$ on ${_p\pi_*^s}$ falls into the submodule of ${_p\pi_*^s}$ spanned by the product of positive dimensional elements.
\end{lmm}

\begin{proof}
Let $\Z_{(p)}$ denote the ring of $p$-adic integers. Following Minami \cite{Minami} consider the unit map $S^0\to H\Z_{(p)}$ and define the $p$-local augmentation of $S^0$, denoted by $IS^0$, to be the fibre of this map. The specturm $IS^0$ satisfies ${_p\pi_i}IS^0\simeq{_p\pi_i^s}$ with $i>0$ and ${_p\pi_0IS^0}\simeq 0$. The restriction of the ring spectrum multiplication on $S^0$ provides us with a pairing $IS^0\wedge IS^0\to S^0$. Hence, any element $f\in{_p\pi_i^s}$ with $i>0$ which admits a factorisation through $IS^0\wedge IS^0\to S^0$ must fall into the submodule of ${_p\pi_*^s}$ spanned by the product of positive dimensional elements of ${_p\pi_*^s}$.\\
For $G=O(1)$ and $p=2$ notice that the restriction of the Kahn-Priddy map, or equivalently the $\Z/2$ transfer, $\R P\vee S^0\simeq \R P_+\to S^0$ to the $S^0$ summand is stable map of degree $2$. This together with the fact that $H^0(\R P;\Z_{(2)})\simeq 0$ means that the composition $\R P_+\to S^0\to H\Z_{(2)}$ is null and hence lifts to a map $\widetilde{\lambda}:\R P\to IS^0$. The $n$-fold transfer $\R P^{\wedge n}_+\to S^0$ then does factorise as
$$\R P^{\wedge n}_+\to (IS^0)^{\wedge n}\to (S^0)^{\wedge n}\to S^0.$$
This proves the result in this case. For $G$ a compact Lie group with $\dim\mathfrak{g}>0$, since $\Sigma^{\dim\mathfrak{g}}BG_+$ has no cells in dimension $0$, hence $H^0(\Sigma^{\dim\mathfrak{g}}BG_+;\Z_{(p)})\simeq 0$. By a similar argument the transfer $\Sigma^{\dim\mathfrak{g}}BG_+\to S^0$ admits a lifting to $IS^0$ and consequently the $n$-fold transfer factors through $(IS^0)^{\wedge n}\to S^0$. The result then follows.
\end{proof}

An immediate consequence of the above lemma together with Theorem \ref{main0} is the following.

\begin{crl}
Choosing $G$ and $p$ as the above lemma, the image of the composition
$${_2\pi_*^s}\Sigma^{n\dim\mathfrak{g}}BG_+^{\wedge n}\to{_2\pi_*^s}\simeq{_2\pi_*}Q_0S^0\to H_*(Q_0S^0;\Z/2)$$
is contained within $\Z/p\{h(\eta^2),h(\nu^2),h(\sigma^2)\}$.
\end{crl}

The proof of Theorem \ref{n-foldtransfer} will be complete once we deal with the special case of $G=O(1)$. It was observed by Lin \cite[Theorem 1.1]{Lin} that the Kahn-Priddy map $\R P\to S^0$, or equivalently the transfer associated to $1<O(1)$, induces a map of Adams spectral sequences $\ext^{s,t}_A(\R P)\to E_2^{s+1,t+1}(S^0)$ which is an epimorphism. For the moment, we are more interested in the raise of the filtration which is easier to prove and generalise. We have the following.

\begin{prp}
Suppose $n>1$ and $p$ is a prime. Let $G=O(1)$ and $p=2$, or $G$ any compact Lie group with $\dim\mathfrak{g}>0$ and $p$ any prime.
Then the $n$-fold transfer $\Sigma^{n\dim\mathfrak{g}}BG_+^{\wedge n}\to S^0$ induces a homomorphism of Adams spectral sequences which raises the Adams filtration by $n$, i.e.
$$\mathrm{Ext}^{s,t}_{A_p}(\widetilde{H}^*\Sigma^{n\dim\mathfrak{g}}BG_+^{\wedge n},\Z/p)\to \mathrm{Ext}^{s+n,t+n}_{A_p}(\Z/p,\Z/p).$$
\end{prp}
\begin{proof}
We deal with the case of $G=O(1)$ and $p=2$ and the other cases are similar. Write $\lambda:\R P_+\to S^0$ for the transfer associated to $\iota:1<O(1)$, or the map of Kahn-Priddy. First, note that the the monomorphism $1<O(1)^{\times n}$ factorises as
$$1<1^{\times n}<1^{\times (n-1)}\times O(1)<1^{\times (n-2)}\times O(1)\times O(1)<\cdots<O(1)^{\times n}$$
so the associated transfer admits a factorisation. Writing $\iota:1\to O(1)$ for the usual embedding, the embedding $1^{\times(n-k)}\times O(1)^{\times k}<1^{\times (n-k-1)}\times O(1)\times O(1)^{\times k}$ is given by $1_{n-k-1}\times\iota\times 1_{O(1)^{\times k}}$ . Therefore, the associated transfer map is given by
$$1\wedge\lambda\wedge 1:B1^{\times(n-k-1)}_+\wedge BO(1)_+\wedge BO(1)^{\times k}_+\to B1^{\times(n-k-1)}_+\wedge B1_+\wedge BO(1)^{\times k}_+.$$

Composition of these transfer maps when $k$ varies, followed by the product $(B1^{\times n})_+=(S^0)^{\wedge n}\to S^0$ is equals to $\lambda^{\wedge n}$ which is compatible with the fact that the transfer associated to $1<O(1)^{\times n}$ is $\lambda^{\wedge n}$. The fact that the restriction of $\lambda$ to the base point provides a stable map $S^0\to S^0$ together with $H^0(\R P;\Z/2)\simeq 0$ implies that $\lambda:\R P\to S^0$ lifts to the first stage of the Adams resolution for $S^0$
$$\xymatrix{
&Y_1\ar[d]\\
\R P\ar[r]^{\lambda}\ar[ru]^{\overline{\lambda}} & S^0\ar[r] & H\Z/2}$$
where we write $Y_i(X)$ for the $i$-th stage of the mod $p$ Adams resolution for $X$, and $Y_i=Y_i(S^0)$. Now, by applying $-\wedge\R P$ to the Adams resolution for $S^0$, noting that the first stage of the Adams resolution for $\R P$, denote by $Y_1(\R P)$ is the fibre of $S^0\wedge \R P\to H\Z/2\wedge\R P$ the map $1\wedge\lambda$ in the following diagram admits an obvious lifting
$$\xymatrix{
& Y_1(\R P)\ar[d]\\
\R P\wedge\R P\ar[ru]^-{\overline{\lambda\wedge 1}}\ar[r]^-{\lambda \wedge 1} & \R P\wedge S^0\ar[r] & H\Z/2\wedge \R P.}$$
Now, the composition $\R P^{\wedge 2}\to Y_1(\R P)\to Y_2$ induces a map of Adams spectral sequences
$$E_2^{s,t}(\R P^{\wedge 2}_+)=\ext^{s,t}_A(\widetilde{H}^*\R P^{\wedge 2}_+,\Z/2)\to \ext_A^{s+2,t+2}(\Z/2,\Z/2)=E_2^{s+2,t+2}.$$
The cases for $n>2$ are similar, one applies $-\wedge\R P$ to the Adams resolution for $\R P^{\wedge (n-1)}$ and plays the game above to get the desired lifting. We leave this to the reader.
\end{proof}


\begin{proof}[Completing proof of Theorem \ref{n-foldtransfer}]
For $n>2$ by the above theorem any element in the image of ${_2\pi_*^s}\Sigma^{n\dim\mathfrak{g}}BG_+^{\wedge n}\to{_2\pi_*^s}$ will be an elements of Adams filtration at least $3$, and also a decomposable element by Lemma \ref{trans-decomposable}. By Theorem \ref{necessary} the must vanish under the Hurewicz homomorphism. Similarly, if $n=2$ and $s>0$ then the image of $2$-fold transfer consists of decomposable elements of Adams filtration at least $3$. Theorem \ref{necessary} the completes the proof. For $n=2$, $s=0$, by the above lemma, the image of the composition
$${_p\pi_*^s}\Sigma^{n\dim\mathfrak{g}}BG_+^{\wedge n}\to {_p\pi_*^s}\simeq{_p\pi_*}Q_0S^0\to H_*(Q_0S^0;\Z/p)$$
falls into $\Z/p\{h(\eta^2),h(\nu^2),h(\sigma^2)\}$. It is then elementary that in the case of $G=O(1)$ and $p=2$, $h(\eta^2)$, $h(\nu^2)$ and $h(\sigma^2)$ are in the image. For instance, let $\widetilde{\eta}\in \pi_1^s\R P_+$ be a pull back of $\eta\in\pi_1^s$ along $\lambda$. It is standard that $\widetilde{\eta}\wedge \widetilde{\eta}\in\pi_2^s\R P^{\wedge 2}_+$ maps to $\eta^2=\eta\wedge\eta$ under $\lambda^{\wedge 2}$. Hence, $h(\eta^2)$ is in the above composition. The other cases are verified similarly.
\end{proof}

Let's conclude this section by noting that the results of this section have applications to the study of bordism group of immersions as discussed in \cite{AsadiEccles-determining}, especially the bordism group of skew-framed immersions of \cite{AkhmetevEccles} which we postpone to a future work.

\section{Non-factorisation of $\theta_j$ elements for $j>3$}
Let $\theta_j\in{_2\pi_{2^{j+1}-2}^s}$ denote Kervaire invariant one element which after \cite{HillHopkinsRavenel} we know exists for $1\leqslant j\leqslant 5$ and possibly for $j=6$, and not for $j>6$. The results of this section, provide applications for our observations on the $n$-fold transfer maps. The following might be well-known, but we record a proof.
\begin{thm}
For $j>3$, the $\theta_j$ elements are not decomposable in ${_2\pi_*^s}$.
\end{thm}

\begin{proof}
According to Madsen \cite[Theorem 7.3]{Madsenthesis} if there exists a Kervaire invariant one element $\theta_j\in{_2\pi_{2^{j+1}-2}^s}$ then it maps nontrivially under the Hurewicz homomorphism
$${_2\pi_{2^{j+1}-2}^s}\simeq{_2\pi_{2^{j+1}-2}}Q_0S^0\to H_{2^{j+1}-2}(Q_0S^0;\Z/2).$$
Suppose $\theta_j$ is a decomposable with $h(\theta_j)\neq 0$. By Theorem \ref{main0} it must live either in one of the dimensions $2,6,14$. This is not possible for $j>3$. This completes the proof.
\end{proof}

Minami \cite{Minami-Kervaire} has shown that for $j>3$, the $\theta_j$ elements do not factorise through the double transfer $\R P^{\wedge 2}_+\to S^0$. We reprove and generalise this in the following way.

\begin{thm}
For $j>3$ and $n>1$, the Kervaire invariant one elements $\theta_j\in{_2\pi_{2^{j+1}-2}^s}$ do not factor through the $n$-fold transfer $\Sigma^{n\dim\mathfrak{g}}BG^{\wedge n}_+\to S^0$ where $G=O(1)$ or $G$ any compact Lie group with $\dim\mathfrak{g}>0$.
\end{thm}

\begin{proof}
If $\theta_j$ factors throughout the $n$-fold transfer, then by the discussion of Section \ref{Hurewicz-transfer} it has to be decomposable element in ${_2\pi_{2^{j+1}-2}^s}$ which we know is not possible for $j>3$. This completes the proof.
\end{proof}

The above theorem also follows form Theorem \ref{n-foldtransfer}. If $\theta_j$ exists then $h(\theta_j)\neq 0$ whereas by Theorem \ref{n-foldtransfer} it must map trivially under $h$. This gives another short proof of the above theorem.

\end{document}